\documentclass[letterpaper, 10 pt, conference]{ieeeconf}  

\IEEEoverridecommandlockouts                              
\usepackage{dcolumn}
\usepackage{lineno,hyperref}
\usepackage{booktabs}
\usepackage{multirow}
\usepackage{siunitx}
\usepackage{amsmath}%
\usepackage{algorithm}
\usepackage{algorithmicx}
\usepackage{algpseudocode}
\usepackage{graphicx}
\usepackage{subfigure}
\usepackage{amsfonts}

\usepackage{mathrsfs}
\allowdisplaybreaks[4]
\usepackage{autobreak}
\usepackage{rotating} 
\usepackage{threeparttable}

\newtheorem{definition}{Definition} 
\newtheorem{proposition}{Proposition} 
 
\newtheorem{remark}{Remark}

\title{\LARGE \bf
An under-approximation for the robust uncertain two-level cooperative set covering problem}

\author{Shuxin Ding,
	Qi Zhang,
	and Zhiming Yuan
\thanks{This work was supported in part by the National Natural Science Foundation of China under grant 61790575, U1834211, U1934220, and in part by the Foundation of China Academy of Railway Sciences Corporation Limited under Grant 2019YJ071. {\it(Corresponding author: Qi Zhang)}}
\thanks{Shuxin Ding, Qi Zhang, and Zhiming Yuan are with the Signal and Communication Research Institute, China Academy of Railway Sciences Corporation Limited, Beijing 100081, China.
	{\tt\small dingshuxin@rails.cn; zhangqi@rails.cn; zhimingyuan@hotmail.com}}
}

\begin{document}

\maketitle
\thispagestyle{empty}
\pagestyle{empty}

\begin{abstract}

This paper investigates the robust uncertain two-level cooperative set covering problem (RUTLCSCP). Given two types of facilities, which are called y-facility and z-facility. The problem is to decide which facilities of both types to be selected, in order to cover the demand nodes cooperatively with minimal cost. It combines the concepts of robust, probabilistic, and cooperative covering by introducting ``$\Gamma$-robust two-level-cooperative $\alpha$-cover'' constraints. Additionally, the constraint relaxed verison of the RUTLCSCP, which is also a linear approximation robust counterpart version of RUTLCSCP (RUTLCSCP-LA-RC), is developed by linear approximation of the constraints, and can be stated as a compact mixed-integer linear programming problem. We show that the solution for RUTLCSCP-LA-RC, $\varepsilon$-under-approximate solution, can also be the solution for RUTLCSCP on some conditions. Computational experiments show that the solutions in 333 instances (10125 instances in total) with 12 types which tinily violate the constraints of RUTLCSCP, can be an efficient under-approximate solutions, while the feasible solutions in other instances are proven to be optimal.
\end{abstract}


\section{Introduction}

The set covering problem (SCP) is one of the most studied combinatorial optimization problems. In the SCP, a set $\mathcal{I}=\{1,\dots,m\}$ of $m$ demand nodes, a set $\mathcal{J}=\{1,\dots,n\}$ of $n$ potential facility location sites and their building costs $c_j$ are given. The 0-1 matrix $A=[a_{ij}]_{m\times n}$ indicates whether a location $j\in\mathcal{J}$ is able to cover a demand node $i\in\mathcal{I}$. The goal of SCP is to find a minimum cost cover of the demand nodes by $x$ where $x_j$ is a binary value whether site $j$ is seleted. It is proven to be NP-complete \cite{chvatal1979greedy}.

\vspace*{-4mm}
\begin{small}
\begin{align}
\min~&\sum\limits_{j\in\mathcal{J}}c_jx_j&\label{SCP:1}\\
{\rm s.t.}~&\sum\limits_{j\in\mathcal{J}}a_{ij}x_j\geq 1&\forall i \in \mathcal{I}\label{SCP:2}\\
&x_j\in \{0,1\}&\forall j \in \mathcal{J}\label{SCP:3}
\end{align}
\end{small}
\vspace*{-6mm}

The SCP has widely been used in many real-world application, especially in facility locations \cite{farahani2012covering}, where both exact and heuristic algorithms are proposed to deal with it. Daskin \cite{daskin1983maximum} considers the facility may not by working with probability, and it can be applied in many application, e.g., node deployment in wireless sensor networks \cite{ding2014improved}, weapon platforms \cite{ding2015status}, etc. Beraldi \textit{et al.} \cite{beraldi2002probabilistic} proposed the probabilistic set-covering aiming at covering constraint satisfied with a predefined probability. Aardal \textit{et al.} \cite{aardal1996two} considered more than one facility type, and proposed a two-level uncapacitated facility location problem. Berman \textit{et al.} \cite{berman2009cooperative} first proposed the cooperative cover model with one facility type.

Pereira \textit{et al.} \cite{pereira2013robust} proposed the robust SCP with uncertain cost coefficients within predefined interval. To the best of our knowledge, the robust set covering problem with probabilistically and cooperative covering by two types of facilities has not previouly analyzed. Xin \textit{et al.} \cite{xin2018efficient} discussed the sensor-weapon-target assignment problem as a collaborative task assignment of sensor and weapon platforms. The probability of capturing the target is similar to the cooperative covering in this paper, while the former is regarded as the objective function. We summarize our contributions as follows:
\begin{enumerate}
	\item [1.] A compact mixed-integer linear programming formulation is proposed by utilizing robust optimization and constraint relaxation.	
	\item [2.] The proposed formulation is analyzed on a large set of test cases with 10125 different instances.
	\item [3.] A majority of the under-approximate soloutions are proven to be optimal while few of them slightly violate the constraints and provide an efficient lower bound.
\end{enumerate}

The rest of this paper is organized as follows. Section \ref{sec2} formulates the robust uncertain two-level cooperative set covering problem. Section \ref{sec3} presents some properties of the model. Performance evaluation results are presented and analyzed in Section \ref{sec4}. Conclusions are given in Section \ref{sec5}.

\section{Formulating the Robust Uncertain Two-Level Cooperative Set Covering Problem}\label{sec2}

\subsection{The Deterministic and Uncertain Two-Level Cooperative Set Covering Problem}

In the Two-Level Cooperative Set Covering Problem (TLCSCP), a set $\mathcal{I}=\{1,\dots,m\}$ of $m$ demand nodes, a set $\mathcal{J}=\{1,\dots,n_1\}$ of $n_1$ potential y-facility location sites and a set $\mathcal{K}=\{1,\dots,n_2\}$ of $n_2$ potential z-facility location sites are given. The 0-1 matrix $A=[a_{ij}]_{m\times n_1}$ or $B=[b_{ik}]_{m\times n_2}$ indicates whether a location $j\in\mathcal{J}$ or $k\in\mathcal{K}$ is able to cover a demand node $i\in\mathcal{I}$. $c^1_j$ represents the costs of building y-facility located in site $j$, and $c^2_k$ represents the costs of building z-facility located in site $k$. Both $y_j$ and $z_k$ are binary value, which means whether building a y-facility in site $j$ and z-facility in site $k$. The objective is to find two subsets $\mathcal{C}^1\subseteq\mathcal{J}$ and $\mathcal{C}^2\subseteq\mathcal{K}$ with minimal cost $c(\mathcal{C}^1,\mathcal{C}^2):=\sum_{j\in\mathcal{C}^1}c^1_j+\sum_{k\in\mathcal{C^2}}c^2_k$ covering all the demand nodes, i.e., for each demand node $i\in\mathcal{I}$ there exists at least one y-facility $j\in\mathcal{C}^1$ and z-facility $k\in\mathcal{C}^2$ which ensures $a_{ij}=1$ and $b_{ik}=1$ simultaneously. A standard binary nonlinear programming formation of Two-Level Cooperative Set Covering Problem is defined as

\vspace*{-4mm}
\begin{small}
\begin{align}
\min~&\sum\limits_{j\in\mathcal{J}}c^1_jy_j+\sum\limits_{k\in\mathcal{K}}c^2_kz_k&\label{TLCSCP:1}\\
{\rm s.t.}~&\left(\sum\limits_{j\in\mathcal{J}}a_{ij}y_j\right)\cdot\left(\sum\limits_{k\in\mathcal{K}}b_{ik}z_k\right)\geq 1&\forall i \in \mathcal{I}\label{TLCSCP:2}\\
&y_j\in \{0,1\}&\forall j \in \mathcal{J}\label{TLCSCP:3}\\
&z_k\in \{0,1\}&\forall k \in \mathcal{K}.\label{TLCSCP:4}
\end{align}
\end{small}where Eq. (\ref{TLCSCP:1}) minimize the building cost of two kinds of facilities. Eq. (\ref{TLCSCP:2}) ensures that for each demand node, it is covered at least one y-faciliy and z-facility simultaneously. Eqs. (\ref{TLCSCP:3}) and (\ref{TLCSCP:4}) ensures decision varibles are binary value. Since $a_{ij}$, $b_{ik}$, $y_j$ and $z_k$ are binary value, TLCSCP is equivalent to the following integer linear programming formulation:

\vspace*{-4mm}
\begin{small}
\begin{align}
\min~&\sum\limits_{j\in\mathcal{J}}c^1_jy_j+\sum\limits_{k\in\mathcal{K}}c^2_kz_k\nonumber&\\
{\rm s.t.}~&\sum\limits_{j\in\mathcal{J}}a_{ij}y_j\geq 1&\forall i \in \mathcal{I}\label{TLCSCP:5}\\
&\sum\limits_{k\in\mathcal{K}}b_{ik}z_k\geq 1&\forall i \in \mathcal{I}\label{TLCSCP:6}\\
&y_j\in \{0,1\}&\forall j \in \mathcal{J}\nonumber\\
&z_k\in \{0,1\}&\forall k \in \mathcal{K}\nonumber,
\end{align}
\end{small}where Eqs. (\ref{TLCSCP:5}) and (\ref{TLCSCP:6}) linearize Eq. (\ref{TLCSCP:2}). And similar to Set Covering Problem \cite{lutter2017improved}, Two-Level Cooperative Set Covering Problem is also a NP-hard combinatorial optimizaiton problem.


Then, the Generalized Uncertain Two-Level Cooperative Set Covering Problem (GUTLCSCP) is formulated based on TLCSCP, which introduces uncertainty into covering model. $a_{ij}$ and $b_{ik}$ are independent random binary variable: with a probability of $1-p_{ij}$ when $a_{ij}=1$ and $p_{ij}$ when $a_{ij}=0$; with a probability of $1-q_{ik}$ when $b_{ik}=1$ and $q_{ik}$ when $b_{ik}=0$. Since the probabilities are assumed to be independent, the probability of two sets $\mathcal{C}^1$ and $\mathcal{C}^2$ coopeartively covering demand node $i$ is as follows:

\vspace*{-4mm}
\begin{small}
\begin{align*}
&P\!\left(\sum\limits_{j\in\mathcal{C}^1}a_{ij}\!\geq\! 1\right)\!\!=\!1-\!\!\prod_{j\in\mathcal{C}^1}\!p_{ij},~
P\!\left(\sum\limits_{k\in\mathcal{C}^2}b_{ik}\!\geq\! 1\right)\!\!=\!1-\!\!\prod_{k\in\mathcal{C}^2}\!q_{ik}.\\
&P\!\!\left[\!\!\left(\sum\limits_{j\in\mathcal{C}^1}a_{ij}\!\geq\!1\right)\!\!\cdot\!\!\left(\sum\limits_{k\in\mathcal{C}^2}b_{ik}\!\geq\!1\right)\!\!\right]
\!\!=\!\!\left(\!\!1-\!\!\prod_{j\in\mathcal{C}^1}p_{ij}\!\!\right)\!\!\cdot\!\!\left(\!\!1-\!\!\prod_{k\in\mathcal{C}^2}q_{ik}\!\!\right).
\end{align*}
\end{small}Then, the GUTLCSCP can be formulated as binary model given by

\vspace*{-4mm}
\begin{small}
\begin{align}
\min~&\sum\limits_{j\in\mathcal{J}}c^1_jy_j+\sum\limits_{k\in\mathcal{K}}c^2_kz_k&\nonumber\\
{\rm s.t.}~&P\left(\sum\limits_{j\in\mathcal{J}}a_{ij}y_j\geq 1,\sum\limits_{k\in\mathcal{K}}b_{ik}z_k\geq 1\right)\geq\alpha&\forall i \in \mathcal{I}\label{GUTLCSCP:1}\\
&y_j\in \{0,1\}&\forall j \in \mathcal{J}\nonumber\\
&z_k\in \{0,1\}&\forall k \in \mathcal{K}\nonumber.
\end{align}
\end{small}

\vspace*{-5mm}
When a solution $y^*\in\{0,1\}^{n_1}$ and $z^*\in\{0,1\}^{n_2}$ is feasible for the GUTLCSCP, Eq. (\ref{GUTLCSCP:1}) is equivalent to

\vspace*{-4mm}
\begin{small}
\begin{align}
\left(1-\prod_{j\in\mathcal{C}^1(y^*)}p_{ij}\right)\cdot\left(1-\prod_{k\in\mathcal{C}^2(z^*)}q_{ik}\right)\geq\alpha\label{GUTLCSCP:3}
\end{align}
\end{small}for all $i\in\mathcal{I}$ with $\mathcal{C}^1(y^*)=\{j\in\mathcal{J}|y^*=1\}$ and $\mathcal{C}^2(z^*)=\{k\in\mathcal{K}|z^*=1\}$. The sets $\mathcal{C}^1(y^*)$ and $\mathcal{C}^2(z^*)$ satifying Eq. (\ref{GUTLCSCP:3}) is referred as two-level-cooperative $\alpha$-cover.

Therefore, GUTLCSCP can be reformulated as:

\vspace*{-4mm}
\begin{small}
\begin{align}
\min~&\sum\limits_{j\in\mathcal{J}}c^1_jy_j+\sum\limits_{k\in\mathcal{K}}c^2_kz_k&\nonumber\\
{\rm s.t.}~&\left(1-\prod_{j\in\mathcal{J}}p_{ij}^{y_j}\right)\cdot\left(1-\prod_{k\in\mathcal{K}}q_{ik}^{z_k}\right)\geq\alpha&\forall i\in \mathcal{I}\label{GUTLCSCP:2}\\
&y_j\in \{0,1\}&\forall j \in \mathcal{J}\nonumber\\
&z_k\in \{0,1\}&\forall k \in \mathcal{K}\nonumber.
\end{align}
\end{small}where Eq. (\ref{GUTLCSCP:2}) is a nonlinear constraint. A linear approximation method is given as follows.

In Eq. (\ref{GUTLCSCP:2}), set $m_i=\prod_{j\in\mathcal{J}}p_{ij}^{y_j}$, $n_i=\prod_{k\in\mathcal{K}}q_{ik}^{z_k}$, $(1-m_i)(1-n_i)\geq\alpha$ for all $i\in\mathcal{I}$. The original constraint Eq. (\ref{GUTLCSCP:2}) can be reformulated as:
\begin{align}
\begin{cases}
m_i = \prod_{j\in\mathcal{J}}p_{ij}^{y_j}\\
n_i = \prod_{k\in\mathcal{K}}q_{ik}^{z_k}\\
(1-m_i)(1-n_i)\geq\alpha
\end{cases}\!\!\!
{\Longleftrightarrow}
\begin{cases}
\ln(m_i) = \sum\nolimits_{j\in\mathcal{J}}\ln(p_{ij})y_j\\
\ln(n_i) = \sum\nolimits_{k\in\mathcal{K}}\ln(q_{ik})z_k\\
(1-m_i)(1-n_i)\geq\alpha
\end{cases}\!\!\!\!\!\!\!\!\!,
\end{align}
where for all $i\in\mathcal{I}$ with $m_i$, $n_i$, $\alpha \in [0,1]$. Therefore, GUTLCSCP can be reformulated as:

\vspace*{-4mm}
\begin{small}
\begin{align}
\min~&\sum\limits_{j\in\mathcal{J}}c^1_jy_j+\sum\limits_{k\in\mathcal{K}}c^2_kz_k&\nonumber\\
{\rm s.t.}~&\ln(m_i) = \sum\limits_{j\in\mathcal{J}}\ln(p_{ij})y_j&\forall i\in\mathcal{I}\label{GUTLCSCP:c1}\\
&\ln(n_i) =\sum\limits_{k\in\mathcal{K}}\ln(q_{ik})z_k&\forall i\in\mathcal{I}\label{GUTLCSCP:c2}\\
&(1-m_i)(1-n_i)\geq\alpha&\forall i\in\mathcal{I}\label{GUTLCSCP:c3}\\
&y_j\in \{0,1\}&\forall j \in \mathcal{J}\nonumber\\
&z_k\in \{0,1\}&\forall k \in \mathcal{K}\nonumber\\
&0\leq m_i\leq1&\forall i\in\mathcal{I}\\
&0\leq n_i\leq1&\forall i\in\mathcal{I}.
\end{align}
\end{small}

\vspace*{-4mm}
The key issue is to deal with the constraint (\ref{GUTLCSCP:c3}). Since $m_i$, $n_i$, $\alpha \in [0,1]$, constraint (\ref{GUTLCSCP:c3}) can be reformulated in part as follows.
\begin{align}
\begin{cases}
1-m_i\geq\alpha\\
1-n_i\geq\alpha
\end{cases}\!\!\!\!\!\!\!\!\Longleftrightarrow
\begin{cases}
\sum\nolimits_{j\in\mathcal{J}}\ln(p_{ij})y_j\leq\ln(1-\alpha)\\
\sum\nolimits_{k\in\mathcal{K}}\ln(q_{ik})z_k\leq\ln(1-\alpha)
\end{cases}\label{GUTLCSCP:c4}
\end{align}
for all $i\in\mathcal{I}$.

According to the Eqs. (\ref{GUTLCSCP:c1}) and 
(\ref{GUTLCSCP:c2}), $\ln(m_i)$ and $\ln(n_i)$ are linear functions with respect to $y_j$ and $z_k$. Besides, we have obtained Eq. (\ref{GUTLCSCP:c4}). As a result, we can construct constraints like
\begin{align}
\beta\ln(m_i)+\gamma\ln(n_i)\leq \ln(F_i(\alpha,\beta,\gamma))\label{GUTLCSCP:c5},
\end{align}
where $\beta+\gamma=1$, $F_i(\alpha,\beta,\gamma)$ is a function with respect to $\alpha$, $\beta$ and $\gamma$ for all $i\in\mathcal{I}$. In order to determine the parameters of the constraint (\ref{GUTLCSCP:c5}), we can find the function tangent to constraint (\ref{GUTLCSCP:c3}).

The constraint (\ref{GUTLCSCP:c5}) can be reformulated as follows

\vspace*{-4mm}
\begin{small}
\begin{align}
\beta\ln(m_i)\!+\!\gamma\ln(n_i)\!\leq\! \ln(F_i(\alpha,\beta,\gamma)){\Longleftrightarrow}
m_i^\beta n_i^\gamma\!\leq\! F_i(\alpha,\beta,\gamma)\label{GUTLCSCP:c7}.
\end{align}
\end{small}

Set $(1-m_i)(1-n_i)=\alpha$, then $m_i=1-\alpha/(1-n_i)$. We can substitute it into Eq. (\ref{GUTLCSCP:c7}) and obtain
\begin{align}
f(n_i)=\left(1-\frac{\alpha}{1-n_i}\right)^\beta n_i^\gamma\label{GUTLCSCP:c8}.
\end{align}
Then by determining the first derivative of Eq. (\ref{GUTLCSCP:c8}), which is $f'(n_i)=0$, the tangency function is obtained.

The solution to $f'(n_i)=0$ is as follows
\begin{align*}
\begin{cases}
n_{i1}=0\\
n_{i2}=1-\alpha\\
n_{i3}=\frac{2\gamma + \alpha\beta - \alpha\gamma - \sqrt{\alpha(4\beta\gamma + \alpha\beta^2 + \alpha\gamma^2 - 2\alpha\beta\gamma)}}{2\gamma}\\
n_{i4}=\frac{2\gamma + \alpha\beta - \alpha\gamma + \sqrt{\alpha(4\beta\gamma + \alpha\beta^2 + \alpha\gamma^2 - 2\alpha\beta\gamma)}}{2\gamma}
\end{cases}.
\end{align*}

The corresponding tangency function is obtained when $n_{i3}$ is selected. Fig. \ref{GUTLCSCP_1} shows comparison between constraint (\ref{GUTLCSCP:c5}) after linear approximation and nonlinear constraint (\ref{GUTLCSCP:c3}) when $\alpha=0.9$, $\beta=0.5$ and $\gamma=0.5$. The range for x-axis and y-axis are determined by Eq. (\ref{GUTLCSCP:c4}) within $[0,0.1]$. Fig. \ref{GUTLCSCP_1} shows that there exists region between these two constraints with one pair of $\beta/\gamma$. Therefore, multiple combination of $\beta/\gamma$ are needed. Fig. \ref{GUTLCSCP_3} shows the comparison when $\beta=[0.1~0.3~0.5~0.7~0.9]$. Fig. \ref{GUTLCSCP_2} combines these $\beta/\gamma$ together to get a intersection of those constraints. Obviously, the linear approximate constraints show great similarity to the original nonlinear constraint. The decision space under the linear approximate constraints is slightly bigger than under nonlinear constraint. Relaxing the problem in Eq. (\ref{GUTLCSCP:2}) leads to the following linear approximation formulation of the GUTLCSCP (GUTLCSCP-LA):

\vspace*{-4mm}
\begin{small}
\begin{align*}
\min~&\sum\limits_{j\in\mathcal{J}}c^1_jy_j+\sum\limits_{k\in\mathcal{K}}c^2_kz_k&\nonumber\\
{\rm s.t.}~&\sum\limits_{j\in\mathcal{J}}\ln(p_{ij})y_j\leq\ln(1-\alpha)&\forall i\in\mathcal{I}\nonumber\\
&\sum\limits_{k\in\mathcal{K}}\ln(q_{ik})z_k\leq\ln(1-\alpha)&\forall i\in\mathcal{I}\nonumber\\
&\beta\sum\limits_{j\in\mathcal{J}}\ln(p_{ij})y_j+\gamma\sum\limits_{k\in\mathcal{K}}\ln(q_{ik})z_k\nonumber\\
&\leq \ln\left[\left(1-\frac{\alpha}{1-\delta}\right)^\beta n_i^\gamma\right]&\forall i\in\mathcal{I}\nonumber\\
&y_j\in \{0,1\}&\forall j \in \mathcal{J}\nonumber\\
&z_k\in \{0,1\}&\forall k \in \mathcal{K}\nonumber,
\end{align*}
\end{small}where $\delta=\frac{2\gamma + \alpha\beta - \alpha\gamma - \sqrt{\alpha(4\beta\gamma + \alpha\beta^2 + \alpha\gamma^2 - 2\alpha\beta\gamma)}}{2\gamma}$. $\beta$, $\gamma\in [0,1]$ are constants or vectors with $\beta+\gamma=1$.

\begin{figure}[!htp]
	\centering
	\subfigure[$\beta=0.5$, $\gamma=0.5$]{
		\includegraphics[width=0.23\textwidth]{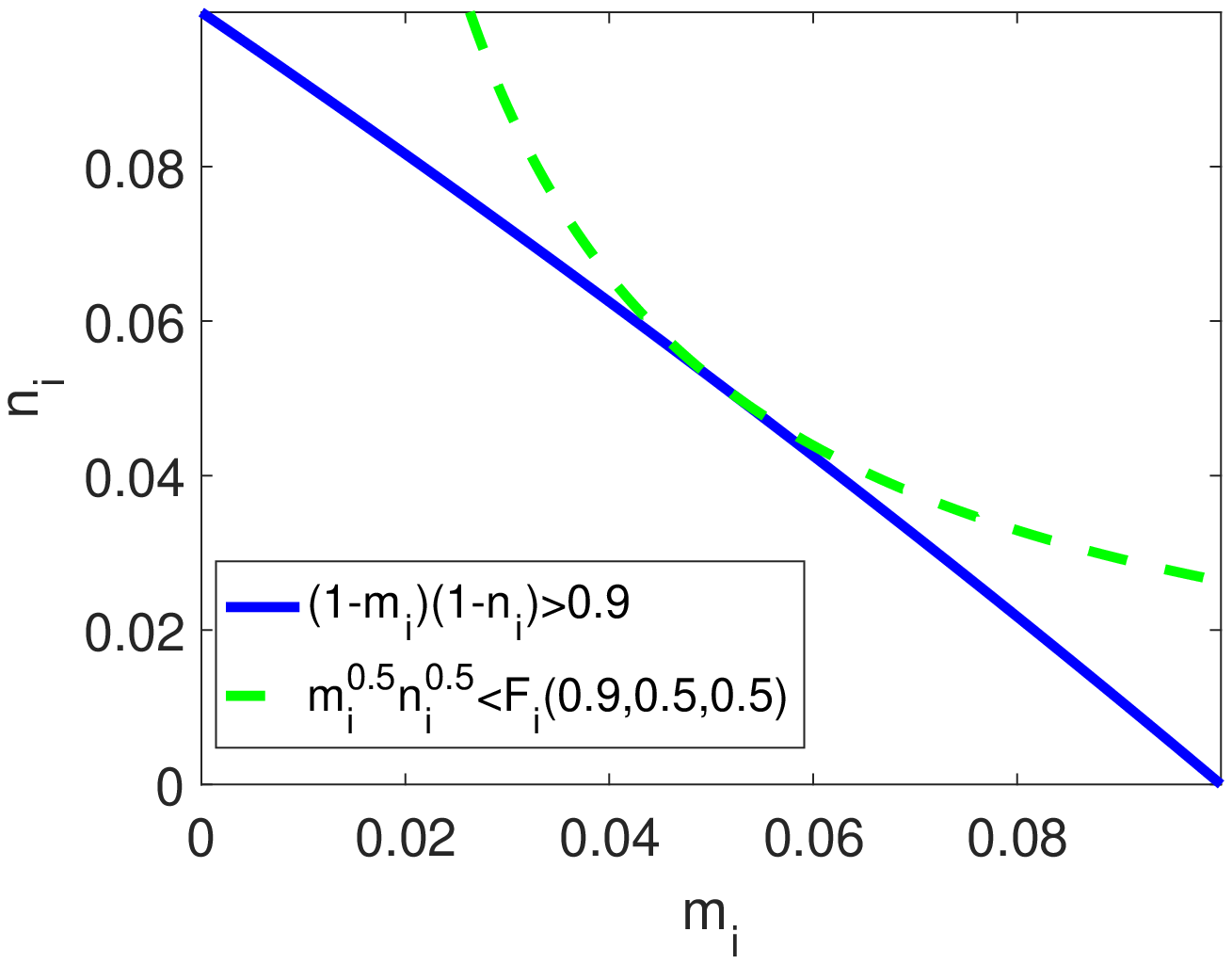}\label{GUTLCSCP_1}}
	\subfigure[multiple constraints with different $\beta/\gamma$  ]{
		\includegraphics[width=0.23\textwidth]{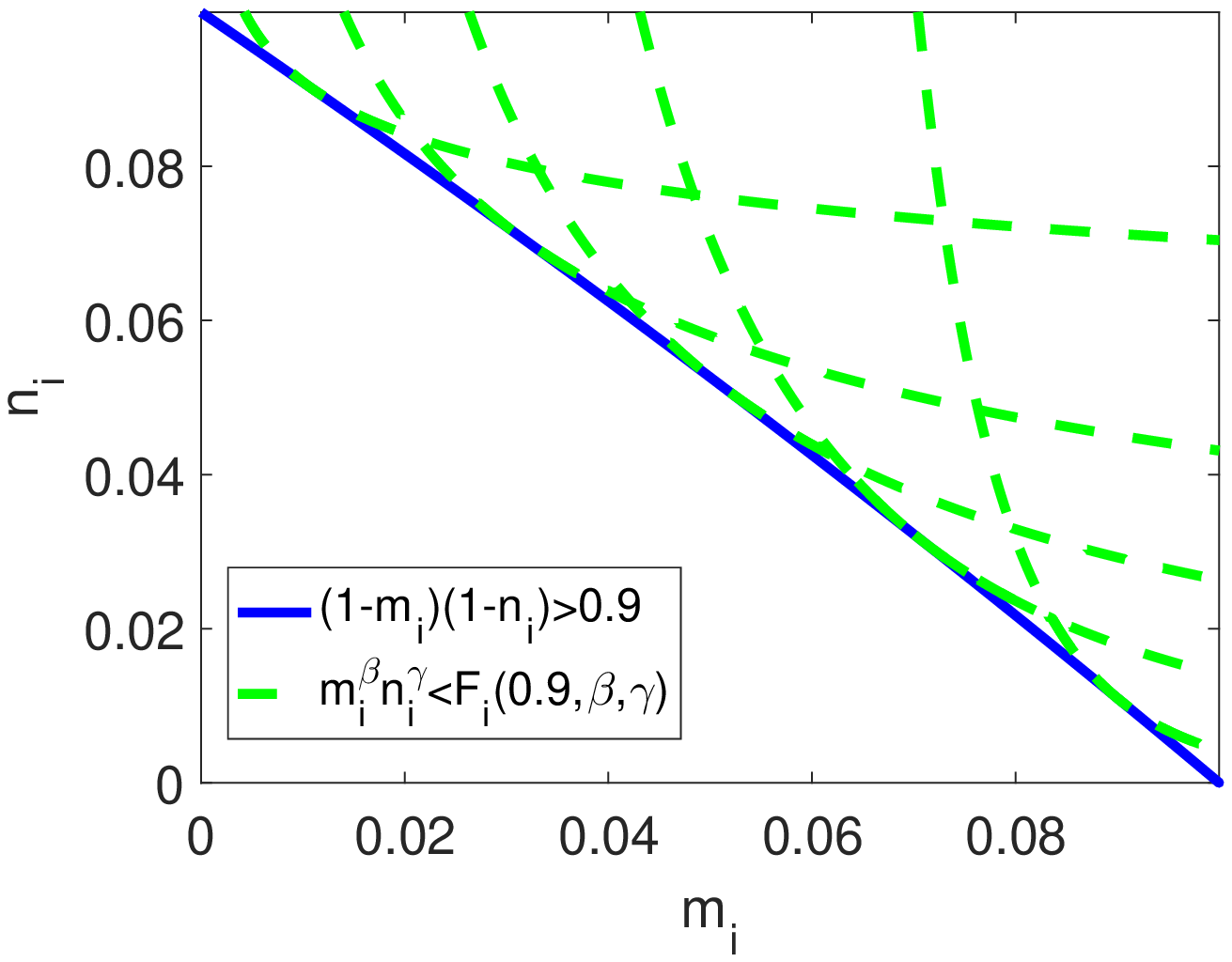}\label{GUTLCSCP_3}}
	\subfigure[multiple constraints with different $\beta/\gamma$ after combination]{
		\includegraphics[width=0.23\textwidth]{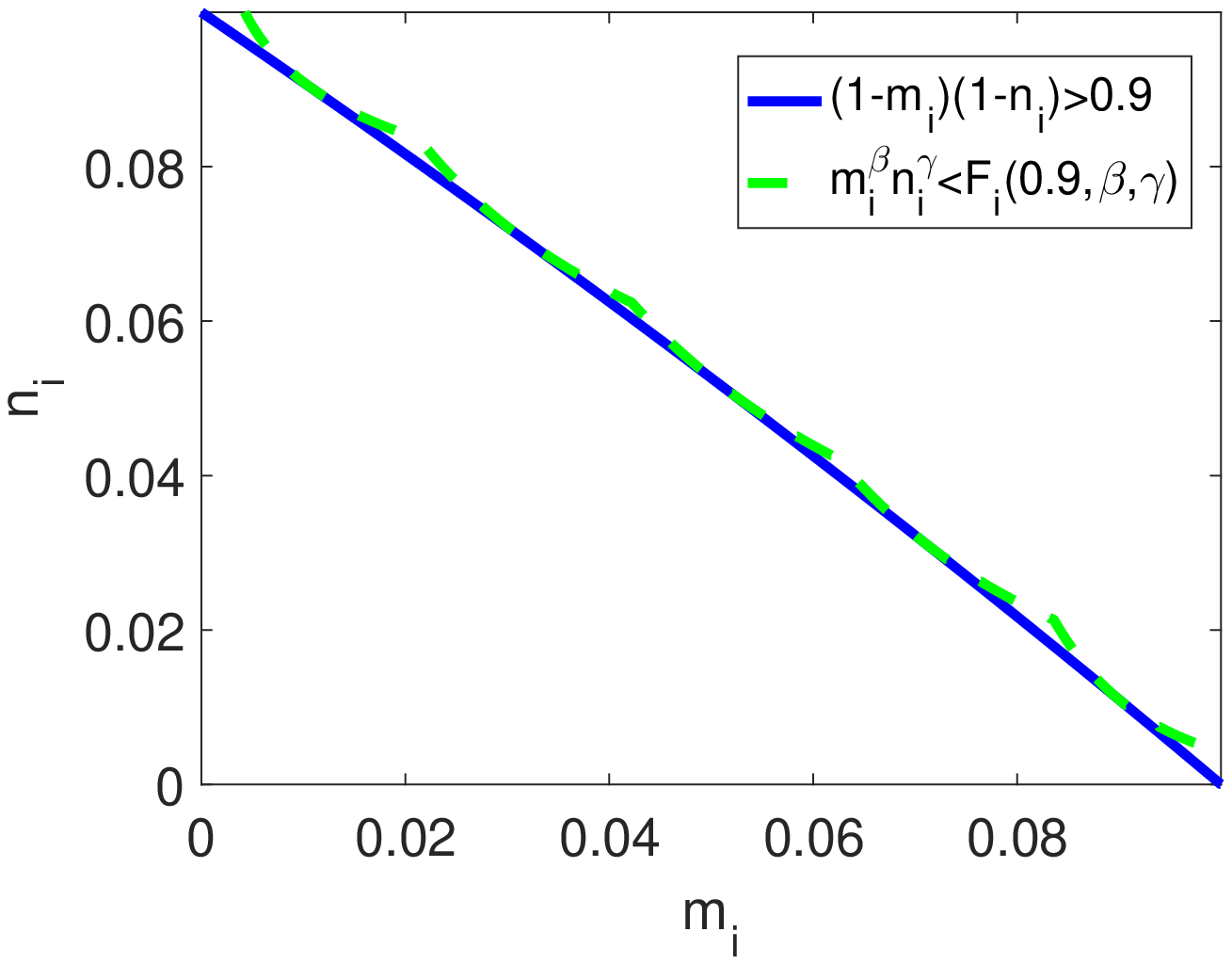}\label{GUTLCSCP_2}}
	\caption{Comparison between constraint (\ref{GUTLCSCP:c5}) after linear approximation and nonlinear constraint (\ref{GUTLCSCP:c3}) when $\alpha=0.9$.}\label{GUTLCSCP}
\end{figure}

\begin{remark}\label{rmk1}
	The linear approximation (LA) method used here actually transforms the original problem to a constraint relaxed problem as an interger linear programming program. All the problems in this paper with LA are the constraint relaxed version of the original problems.
\end{remark}

\subsection{Modeling the Robust Uncertain Two-Level Cooperative Set Covering Problem}
\label{sec3:2}

The Robust Uncertain Two-Level Cooperative Set Covering Problem (RUTLCSCP) is formulated based on the GUTLCSCP, with fluctuation of the probabilities $p_{ij}$ and $q_{ik}$.

In real-world applications, the probabilities $p_{ij}$ and $q_{ik}$ are not precisely known \cite{lutter2017improved}. They can be estimated based on historical data. However, these estimated values could not reflect the whole situation. In some situation, estimated values may be too optimistic, while in other situation, they may be too pessimistic. Hence, there exists a natural fluctuation of the probabilities. Therefore, in order to model the effect of these fluctuations, an interval is established based on the \textit{nominal value}. This interval covers the range of the probabilities. As a result, this description of probability is more reasonable than using a particular value \cite{lutter2017improved}.

The following description is based on a $\Gamma$-scenario set proposed by Bertsimas and Sim \cite{bertsimas2004price}. There are at most $\Gamma$ values deviate from their \textit{nominal value}. When $\Gamma=n$, all parameters are allowed to deviate, which is equivalent to Soyster's robust model \cite{soyster1973convex}. However, this model is too conservative. $\Gamma$ models the risk attitute of the parameters \cite{lutter2017improved}, and it is also called the budget of uncertainty.

We assume that $p_{ij}$ and $q_{ik}$ are uncertain variable within the interval $[\bar{p}_{ij},\bar{p}_{ij}+\hat{p}_{ij}]\subseteq[0,1]$ and $[\bar{q}_{ik},\bar{q}_{ik}+\hat{q}_{ik}]\subseteq[0,1]$ where $\bar{p}_{ij}\geq0$ and $\bar{q}_{ik}\geq0$ are the \textit{nominal value}, $\hat{p}_{ij}\geq0$ and $\hat{q}_{ik}\geq0$ are the \textit{worst case deviation}. The two $\Gamma$-scenairo sets are given by

\vspace*{-4mm}
\begin{small} 
\begin{align*}
&\mathscr{U}^{\Gamma_i}_1\!:=\!\left\{p_{i:}|\forall j\in\mathcal{J}:p_{ij}\in[\bar{p}_{ij},\bar{p}_{ij}\!+\!\hat{p}_{ij}],\sum\limits_{j\in\mathcal{J}}\frac{p_{ij}-\bar{p}_{ij}}{\hat{p}_{ij}}\leq\Gamma_i\!\right\}\\
&\mathscr{U}^{\Gamma_i}_2\!:=\!\left\{q_{i:}|\forall k\in\mathcal{K}:q_{ik}\in[\bar{q}_{ik},\bar{q}_{ik}\!+\!\hat{q}_{ik}],\sum\limits_{k\in\mathcal{K}}\frac{q_{ik}-\bar{q}_{ik}}{\hat{q}_{ik}}\leq\Gamma_i\!\right\}\!\!,
\end{align*}
\end{small}for all $i\in\mathcal{I}$, where $p_{i:}:=(p_{ij})_{j\in\mathcal{J}}$, $q_{i:}:=(q_{ik})_{k\in\mathcal{K}}$.

The difference between RUTLCSCP and GUTLCSCP is that for any $i\in\mathcal{I}$, there exists two-level-cooperative $\alpha$-cover in RUTLCSCP with probabilities satisfying $p_{i:}\in\mathscr{U}^{\Gamma_i}_1$ and $q_{i:}\in\mathscr{U}^{\Gamma_i}_2$. We can consider the worst case: there exists $\Gamma_i\in\mathbb{N}_0$ entries in $p_{i:}$ and $q_{i:}$ derive from their nominal value, which are worst case derivation. The other entries in $p_{i:}$ and $q_{i:}$ are their nominal values $\bar{p}_{ij}$ and $\bar{q}_{ik}$. A $\Gamma$-robust two-level-cooperative $\alpha$-cover is defined as follows.

\begin{definition}
	($\Gamma$-robust two-level-cooperative $\alpha$-cover). Set $i\in\mathcal{I}$, $\Gamma_i\in\mathbb{N}_0$, $\Gamma=(\Gamma_i)_{i\in\mathcal{I}}$, $\alpha\in[0,1)$.  For all $j\in\mathcal{J}$ and $k\in\mathcal{K}$, $p_{ij}$ are within range $[\bar{p}_{ij},\bar{p}_{ij}+\hat{p}_{ij}]\subseteq[0,1]$, $q_{ik}$ are within range $[\bar{q}_{ik},\bar{q}_{ik}+\hat{q}_{ik}]\subseteq[0,1]$.  The worst-case coverage probability for set $C^1\subseteq\mathcal{J}$ and set $C^2\subseteq\mathcal{K}$ can be defined by
	
	\vspace*{-4mm}
	\begin{small}
		\begin{align*}
		&P_{\Gamma_i}\!\!\!\left(\sum\limits_{j\in\mathcal{C}^1}\!a_{ij}\!\geq1\!\right)\!\!:=1-\!\!\!\!\!\!\mathop{\max}_{\{\mathcal{U}_1\subseteq\mathcal{C}^1:|\mathcal{U}_1|\leq\Gamma_i\}}\left\{\prod\limits_{j\in\mathcal{U}_1}(\bar{p}_{ij}\!+\!\hat{p}_{ij})\!\cdot\!\!\!\!\!\!\!\prod\limits_{j\in\mathcal{C}^1\backslash\mathcal{U}_1}\!\!\!\bar{p}_{ij}\right\}\\
		&P_{\Gamma_i}\!\!\!\left(\sum\limits_{k\in\mathcal{C}^2}\!b_{ik}\!\geq1\!\right)\!\!:=1-\!\!\!\!\!\!\mathop{\max}_{\{\mathcal{U}_2\subseteq\mathcal{C}^2:|\mathcal{U}_2|\leq\Gamma_i\}}\left\{\prod\limits_{k\in\mathcal{U}_2}(\bar{q}_{ik}\!+\!\hat{q}_{ik})\!\cdot\!\!\!\!\!\!\!\prod\limits_{k\in\mathcal{C}^2\backslash\mathcal{U}_2}\!\!\!\bar{q}_{ik}\right\}.
		\end{align*}
	\end{small}
	
	\vspace*{-4mm}
	A $\Gamma$-robust two-level-cooperative $\alpha$-cover with $C^1\subseteq\mathcal{J}$ and $C^2\subseteq\mathcal{K}$ have a worst-case coverage probability $P_{\Gamma_i}\left(\sum\nolimits_{j\in\mathcal{C}^1}a_{ij}\geq1\right)\cdot P_{\Gamma_i}\left(\sum\nolimits_{k\in\mathcal{C}^2}b_{ik}\geq1\right)$ greater or equals to $\alpha$. When all $i\in\mathcal{I}$ for set $\mathcal{C}^1$ and set $\mathcal{C}^2$ satisfying $\Gamma$-robust two-level-cooperative $\alpha$-cover, then a $\Gamma$-robust two-level-cooperative $\alpha$-cover is obtained.
\end{definition}

The RUTLCSCP is to find a $\Gamma$-robust two-level-cooperative $\alpha$-cover of minimum costs. A nonlinear formulation can be defined in the following:

\vspace*{-4mm}
\begin{small}
	\begin{align}
	\min~&\sum\limits_{j\in\mathcal{J}}c^1_jy_j+\sum\limits_{k\in\mathcal{K}}c^2_kz_k&\nonumber\\
	{\rm s.t.}~&P_{\Gamma_i}\left(\sum\limits_{j\in\mathcal{J}}a_{ij}y_j\geq1\right)\cdot P_{\Gamma_i}\left(\sum\limits_{k\in\mathcal{K}}b_{ik}z_k\geq1\right)\geq\alpha&\forall i \in \mathcal{I}\label{RUTLCSCP:1}\\
	&y_j\in \{0,1\}&\forall j \in \mathcal{J}\nonumber\\
	&z_k\in \{0,1\}&\forall k \in \mathcal{K}\nonumber.
	\end{align}
\end{small}
A solution $y^*\in\{0,1\}^{n_1}$, $z^*\in\{0,1\}^{n_2}$ is called robust feasible when $\Gamma$-robust two-level-cooperative $\alpha$-cover is satisfied. There exists two maximum subproblems in Eq. (\ref{RUTLCSCP:1}) defined as

\vspace*{-6mm}
\begin{small}
	\begin{align}
	\beta^1_i(y,\Gamma_i):=\mathop{\max}_{\{\mathcal{U}_1\subseteq\mathcal{C}^1(y):|\mathcal{U}_1|\leq\Gamma_i\}}\left\{\prod\limits_{j\in\mathcal{U}_1}(\bar{p}_{ij}\!+\!\hat{p}_{ij})^{y_j}\cdot\!\!\prod\limits_{j\in\mathcal{J}\backslash\mathcal{U}_1}\!\bar{p}_{ij}^{y_j}\right\}\label{RUTLCSCP:2}\\
	\beta^2_i(z,\Gamma_i):=\mathop{\max}_{\{\mathcal{U}_2\subseteq\mathcal{C}^2(z):|\mathcal{U}_2|\leq\Gamma_i\}}\left\{\prod\limits_{k\in\mathcal{U}_2}(\bar{q}_{ik}\!+\!\hat{q}_{ik})^{z_k}\cdot\!\!\prod\limits_{k\in\mathcal{K}\backslash\mathcal{U}_2}\!\bar{q}_{ik}^{z_k}\right\}\label{RUTLCSCP:3}
	\end{align}
\end{small}where for all $i\in\mathcal{I}$. For a given solution $y^*\in\{0,1\}^{n_1}$, $z^*\in\{0,1\}^{n_2}$.

Therefore, the RUTLCSCP can be reformulated as 

\vspace*{-4mm}
\begin{small}
\begin{align}
\min~&\sum\limits_{j\in\mathcal{J}}c^1_jy_j+\sum\limits_{k\in\mathcal{K}}c^2_kz_k&\nonumber\\
{\rm s.t.}~&\left[1-\beta^1_i(y,\Gamma_i)\right]\cdot\left[1-\beta^2_i(z,\Gamma_i)\right]\geq\alpha&\forall i \in \mathcal{I}\label{RUTLCSCP:4}\\
&y_j\in \{0,1\}&\forall j \in \mathcal{J}\nonumber\\
&z_k\in \{0,1\}&\forall k \in \mathcal{K}\nonumber.
\end{align}
\end{small}

\vspace*{-4mm}
Similarily, we can develop the linear approximate model of the RUTLCSCP based on the GUTLCSCP-LA. Meanwhile, applying the strong duality theorem, we can develop the robust counterpart (RC) of the robust model RUTLCSCP-LA-RC, which is a compact mixed-integer linear programming problem:

\vspace*{-4mm}
\begin{small}
	\begin{align*}
	\min~\!\!\!&\sum\limits_{j\in\mathcal{J}}c^1_jy_j+\sum\limits_{k\in\mathcal{K}}c^2_kz_k&\\
	{\rm s.t.}~&\sum\limits_{j\in\mathcal{J}}\ln(\bar{p}_{ij})y_j\!+\!\sum\limits_{j\in\mathcal{J}}\zeta_{ij}^1+\Gamma_i\eta_i^1\!\leq\ln(1-\alpha)&\forall i\in\mathcal{I}\\
	&\sum\limits_{k\in\mathcal{K}}\ln(\bar{q}_{ik})z_k\!+\!\sum\limits_{k\in\mathcal{K}}\zeta_{ik}^2+\Gamma_i\eta_i^2\!\leq\ln(1-\alpha)&\forall i\in\mathcal{I}\\
	&\beta\left[\sum\limits_{j\in\mathcal{J}}\ln(\bar{p}_{ij})y_j+\sum\limits_{j\in\mathcal{J}}\zeta_{ij}^1+\Gamma_i\eta_i^1\right]&\\
	&+\gamma\left[\sum\limits_{k\in\mathcal{K}}\ln(\bar{q}_{ik})z_k+\sum\limits_{k\in\mathcal{K}}\zeta_{ik}^2+\Gamma_i\eta_i^2\right]&\\
	&\leq \ln\left[\left(1-\frac{\alpha}{1-\delta}\right)^\beta n_i^\gamma\right]&\forall i\in\mathcal{I}\\
	&\zeta_{ij}^1+\eta_i^1\geq\left(\ln(\bar{p}_{ij}+\hat{p}_{ij})-\ln(\hat{p}_{ij})\right)y_j&\forall i\in\mathcal{I},j\in\mathcal{J}\\
	&\zeta_{ik}^2+\eta_i^2\geq\left(\ln(\bar{q}_{ik}+\hat{q}_{ik})-\ln(\hat{q}_{ik})\right)z_k&\forall i\in\mathcal{I},k\in\mathcal{K}\\
	&\zeta_{ij}^1\geq0&\forall i\in\mathcal{I},j\in\mathcal{J}\\
	&\zeta_{ik}^2\geq0&\forall i\in\mathcal{I},k\in\mathcal{K}\\
	&\eta_i^1\geq0&\forall i\in\mathcal{I}\\
	&\eta_i^2\geq0&\forall i\in\mathcal{I}\\
	&y_j\in \{0,1\}&\forall j \in \mathcal{J}\\
	&z_k\in \{0,1\}&\forall k \in \mathcal{K},
	\end{align*}
\end{small}where $\delta=\frac{2\gamma + \alpha\beta - \alpha\gamma - \sqrt{\alpha(4\beta\gamma + \alpha\beta^2 + \alpha\gamma^2 - 2\alpha\beta\gamma)}}{2\gamma}$. $\beta$, $\gamma\in [0,1]$ are constants or vectors with $\beta+\gamma=1$.

\section{Properties of the Model}\label{sec3}

There exists nonlinear, noncompact constraints and maximum subproblems in the proposed RUTLCSCP, which are hard to slove. A definition and two propositions are provided as follows.

\begin{definition}\label{definition 1}
	($\varepsilon$-under-approximate solution).  Given a scalar $\varepsilon>0$, a $\varepsilon$-under-approximate solution has a larger feasible region with constraints relaxed than the original feasible region with the original constraints. The new feasible region is obtained by linear approximation of the nonlinear constraints, i.e., $X_{LA}\in\Omega_{LA}=\{x|C_i(X)(1+\varepsilon)\geq\alpha, X=\mathop{\arg\min}_{X\in\Omega}\{F(X)\}, i\in\mathcal{I}\}$, where $\Omega$ is the feasible region of the original problem and $\Omega_{LA}$ is the approximate feasible region.
\end{definition}

\begin{proposition}\label{proposition 1}
	Suppose the solution to the linear approximate problem is $X_{LA}$ with the objective value $F_{LA}(X_{LA})$, while the soluiton for nonlinear constraints problem is $X$ with the objective value $F(X)$. Then we will have $F_{LA}(X_{LA})\leq F(X)$, which is a lower bound on the optimal objective function. If the nonlinear constraints is satisfied when we substitute the solution $X_{LA}$ into the original problem with nonlinear constraints, we will have $F_{LA}(X_{LA})=F(X)$. The nonlinear constraints problems include the GUTLCSCP and the RUTLCSCP, while the linear approximate problem are the GUTLCSCP-LA and the RUTLCSCP-LA-RC.
\end{proposition}

\begin{proof}
The solution to the linear approximate problem is $X_{LA}\in\Omega_{LA}$, while the solution for nonlinear constraints problem is $X\in\Omega$. According to the Fig. \ref{GUTLCSCP_2}, nonlinear constraints are relaxed by linear approximation method. Therefore, $\Omega\in\Omega_{LA}$ is a subset of approximate feasible region. As a result, $F_{LA}(X_{LA})\leq F(X)$.  When the solution $X_{LA}$ satisfies the nonlinear constraints, that means $X_{LA}\in\Omega$. Therefore, we will have $F_{LA}(X_{LA})=F(X)$.
\end{proof}

\begin{proposition}\label{proposition 2}
	If the problem after linear approximation (GUTLCSCP-LA, RUTLCSCP-LA-RC) has no solution, the orginal problem with nonlinear constraints (RUTLCSCP-LA-RC, RUTLCSCP) has no solution as well. 
\end{proposition}

\begin{proof}
Based on Proposition \ref{proposition 1}, we have $\Omega\subseteq\Omega_{LA}$. If there is no solution in the feasible region $\Omega_{LA}$, then there is no solution in the feasible region $\Omega$ as well. In other words, if there is no solution in the linear approximate problem, there is no solution in the original problem with nonlinear constraints.
\end{proof}

Therefore, based on the above propositions, as for problems in different scales, we could use exact method or solver (e.g., IBM-ILOG-CPLEX) to solve the RUTLCSCP-LA-RC in order to obtain the exact solution to the RUTLCSCP if the equality condition in Proposition \ref{proposition 1} is met. Otherwise, $\varepsilon$-under-approximate solution are obtained.

\section{Computational experiments and analysis}\label{sec4}

This section is devoted to the performance investigation of the proposed model.  At first, we present an RUTLCSCP test-case generator which can produce instances of different scales. Then, we solve the problem, which includes exact solutions for RUTLCSCP-LA-RC and approximate solutions for RUTLCSCP. All experiments were carried out on a PC with Intel Xeon E5 CPU 2.60GHz and 64 GB internal memory. RUTLCSCP-LA-RC problems were implemented in MATLAB R2016a using YALMIP as the modeling language and CPLEX 12.5 with default parameter settings.

\subsection{Test-Case Generator}

Due to the lack of benchmark instances for the RUTLCSCP in literature, we consider the following parameter setting. The fix costs coefficients building y-facility $c^1_j$ and z-facility $c^2_k$ were both randomly generated by sampling from a uniform distribution in [0, 100]. The nominal value of probabilities $\bar{p}_{ij}$ and  $\bar{q}_{ik}$ were both obtained by sampling from a uniform distribution in [0.9, 1.0]. Deviations for the default probability $\hat{p}_{ij}$ and $\hat{q}_{ik}$ were both taken from a uniform distribution in [0, 0.1]. Besides, we consider two covering ranges $yr$ and $zr$ for these two kinds of facilities. If the Euclidean distance of the demand node and facility location is greater than the covering range, the corresponding probability ${p}_{ij}$ or ${q}_{ik}$ is 0. Each demand node serves as candidate location site for y-facility and z-facility, i.e., $\mathcal{I}=\mathcal{J}=\mathcal{K}$. The position of the demand nodes are randomly generated within the region $Ax \times Ay$. All the RUTLCSCP formulations were solved for the parameters $\alpha\in\{0.8,0.85,0.9\}$ and $\Gamma\in\{0,\dots,|\mathcal{I}|\}$. 10 cases were considered. For each case, we randomly generated five different instances. In total, 10125 derived RUTLCSCP instances were generated. The detailed information of these instances are in Table \ref{testproblem}.

\vspace*{-2mm}
\begin{table}[!h]
	\setlength{\abovecaptionskip}{0.cm}
	\setlength{\belowcaptionskip}{-0.9cm}
	\renewcommand{\arraystretch}{0.9}
	\centering
	\caption{The test-case for RUTLCSCP} \label{testproblem}
	\resizebox{0.45\textwidth}{!}{\begin{tabular}{cccc}
		\toprule
		Instance		&$(|\mathcal{I}|,|\mathcal{J}|,|\mathcal{K}|)$  & $(yr/km,zr/km)$ & $(A_x/km, A_y/km)$\\
		\midrule
		P1.1--P1.5		&$(20,20,20)$ &$(10,5)$&$(25,25)$\\
		P2.1--P2.5		&$(25,25,25)$ &$(10,5)$&$(25,25)$\\
		P3.1--P3.5		&$(30,30,30)$ &$(10,5)$&$(25,25)$\\
		P4.1--P4.5    	&$(40,40,40)$ &$(14,7)$&$(50,50)$\\
		P5.1--P5.5 	  	&$(50,50,50)$ &$(14,7)$&$(50,50)$\\
		P6.1--P6.5		&$(60,60,60)$ &$(14,7)$&$(50,50)$\\
		P7.1--P7.5     	&$(80,80,80)$ &$(20,10)$&$(100,100)$\\
		P8.1--P8.5		&$(100,100,100)$ &$(20,10)$&$(100,100)$\\
		P9.1--P9.5		&$(120,120,120)$ &$(20,10)$&$(100,100)$\\
		P10.1--P10.5	&$(140,140,140)$ &$(20,10)$&$(100,100)$\\
		\bottomrule
	\end{tabular}}
\end{table}

\vspace*{-4mm}
\subsection{Results and Analysis}

We found that the approximation accuracy of the constraints are related with the amount of the $\beta/\gamma$ pairs. If we use more pairs, the approximation will be better, which increase the total running time of the algorithm. Therefore, one needs to balance these two conflicts. Here we considered the combination $\beta=[0.001~0.01~0.05~0.1~0.15~0.2~0.3~0.4~0.5~0.6~0.7~0.8~0.85~0.9$ $~0.95~0.99~0.999]$ based on empirical testing, where $\gamma=1-\beta$.

The results for RUTLCSCP are presented in Table \ref{result} with the following statistics:
\begin{itemize}
	\item {\it Proof of opt}: The proportion of instances in which the solution was proven to be optimal.
	\item {\it Time}: Arithmetic mean of run times in seconds. 
	\item {\it CV (constraint violation)}: The proportion of violated constraints in RUTLCSCP with feasible $\varepsilon$-under-approximate solution.
	\item {\it Degree of feasibility}: The ratio of feasible solutions without any violated constraint in RUTLCSCP-LA-RC and the total number of instances.
\end{itemize} 

\begin{table*}[!htp]
	\setlength{\abovecaptionskip}{0.0cm}
	\setlength{\belowcaptionskip}{-0.9cm}
	\renewcommand{\arraystretch}{0.9}
	\centering
	\vspace*{2mm}
	\caption{Computational results for RUTLCSCP} \label{result}
	\begin{threeparttable}
		\begin{tabular}{ccccccccccc}
			\toprule
			\multirow{2}{*}{Instance}  & \multicolumn{3}{c}{$\alpha=0.8$\tnote{$^*$}} & \multicolumn{3}{c}{$\alpha=0.85$\tnote{$^*$}} & \multicolumn{4}{c}{$\alpha=0.9$}\\
			\cmidrule(r){2-4}
			\cmidrule(r){5-7}
			\cmidrule(r){8-11}		
			&Opt. (\%) & Time & CV (\%) &Opt. (\%) & Time & CV (\%) &Opt. (\%) & Time & CV (\%) & Degree of feasibility (\%) \\
			\midrule
			P1.1--P1.5		&100.00	  &0.14	  &0.00 &100.00	  &0.17	  &0.00 &100.00	  &0.10	  &0.00	&61.90\\
			P2.1--P2.5		&100.00	  &0.21	  &0.00 &100.00	  &0.25	  &0.00 & 98.75	  &0.20	  &0.05 &61.54\\
			P3.1--P3.5		&100.00	  &0.24	  &0.00 &100.00	  &0.33	  &0.00 & 99.20	  &0.37	  &0.03 &80.65\\
			P4.1--P4.5    	&100.00	  &0.32	  &0.00 &100.00	  &0.48	  &0.00 &100.00	  &0.30	  &0.00 &21.95\\
			P5.1--P5.5 	  	&100.00	  &0.56	  &0.00 &100.00	  &0.76	  &0.00 &100.00	  &0.44	  &0.00 &41.18\\
			P6.1--P6.5		&100.00	  &1.03	  &0.00 &100.00	  &1.01	  &0.00 &100.00	  &0.44	  &0.00 &21.31\\
			P7.1--P7.5     	& 80.25	  &1.48	  &0.25 &100.00	  &1.60	  &0.00 &100.00	  &0.74	  &0.00 & 1.23\\
			P8.1--P8.5		&100.00	  &3.00	  &0.00 & 80.00	  &3.76	  &0.20 & 96.10	  &2.81	  &0.04 &40.59\\
			P9.1--P9.5		&100.00	  &4.59	  &0.00 &100.00	  &5.98	  &0.00 & 99.18	  &4.58	  &0.01 &40.50\\
			P10.1--P10.5	&100.00	  &7.29	  &0.00 & 80.14	  &8.33	  &0.14 &100.00	  &5.68	  &0.00 &20.57\\
			\bottomrule
		\end{tabular}
		\begin{tablenotes}
			\item [$^*$] Degree of feasibility are 100\%.
		\end{tablenotes}
	\end{threeparttable}
\end{table*}

\begin{table}[!htp]
	\setlength{\abovecaptionskip}{0.cm}
	\setlength{\belowcaptionskip}{-0.9cm}
	\renewcommand{\arraystretch}{0.9}
	\centering
	\caption{Instance types with constraint violation} \label{result2}
	\begin{threeparttable}
		\begin{tabular}{cccccc}
			\toprule
			Instance & $\alpha$ & $\Gamma$ &Obj. 	 & $\phi$ 		&	$\#$ \\
			\midrule
			P2.2 	 & 0.9	 &0		&463.94 &7.40E-06     &1/20\\
			P3.1     & 0.9   &0		&309.00 &1.29E-05	  &1/25\\
			P7.1	 & 0.8   &1+	&1464.12&3.79E-04     &1/80\\
			P8.1	 & 0.9	 &0		&1258.93&1.13E-05	  &1/100\\
			P8.2	 & 0.85  &0		&1514.56&4.21E-04	  &1/100\\
			P8.2	 & 0.9   &0		&1603.00&8.01E-06	  &1/100\\
			P8.3	 & 0.85  &1+	&1444.47&2.51E-04$^\ddag$	  &1/100\\
			P8.3	 & 0.9   &0		&1409.12&1.93E-04	  &1/100\\
			P8.3	 & 0.9   &2+	&1942.41$^\dag$ &2.12E-04$^\ddag$&1/100\\ 
			P9.2	 & 0.9   &0		&1501.84&3.91E-04	  &1/120\\
			P9.3	 & 0.9   &0		&1312.91&3.97E-06     &1/120\\
			P10.3	 & 0.85  &1+	&1408.94&3.44E-04	  &1/140\\
			\bottomrule
		\end{tabular}
		\begin{tablenotes}
			\item [$^\dag$]  The objective value are still varying with different $\Gamma$.
			\item [$^\ddag$] The total constraint violations are varying with different $\Gamma$.
		\end{tablenotes}
	\end{threeparttable}
\end{table}

From Table \ref{result}, there exists unfesible solutions for RUTLCSCP-LA-RC since the degree of feasibility is less than 100\%. These instances are especially those with $\alpha=0.9$ and $\Gamma\geq1$. As a result, the corresponding instances of RUTLCSCP have no solution. Besides, for some instances, the solutions violate the original nonlinear constraints but feasible to RUTLCSCP-LA-RC, which are the $\varepsilon$-under-approximate solutions. These instances are shown in Table \ref{result2} with 12 instances types. $\phi$ represents the total constraint violations and $\#$ stands for the proportion of violations with total nonlinear constraints. The solutions for the remaining instances are also the solutions for the original problem RUTLCSCP. Most of the $\phi$ of the approximate solutions are in level of E-4$\sim$E-6, and with only one violated constraint, which means great approximation. The corresponding objective value is closely lower than optimal value, which is an efficient under-approximation and lower bound. 

Due to the limited space for this paper, details of the objective value for each instance are not shown. For $\alpha=0.8$ and $\alpha=0.85$, the objective value are same when $\Gamma\geq1$. However, for $\alpha=0.9$, the objective value are different under different $\Gamma$. Most of the instances have the same objective value when $\Gamma\geq 2$. Noted that in P8.3 ($\alpha=0.9$, $\Gamma\geq2$) marked with $^\dag$, the objective value are still varying when $\Gamma\geq 2$. When $\Gamma=2,3,6,18,19$, the corresponding objective values are 1942.41. The values for the rest are 1947.82. The total constraint violations marked with $^\ddag$ means that they are varying with different $\Gamma$. For example, in P8.3 ($\alpha=0.85$, $\Gamma\geq1$), $\phi=2.51E-04$ when $\Gamma=1$; while $\phi=6.30E-04[1]$ when $\Gamma\geq2$. In P8.3 ($\alpha=0.9$, $\Gamma\geq2$), $\phi=2.12E-04[1]$ when $\Gamma=2,3,6,18,19$; while the rest are constraint satisfied. Noted that CPLEX can efficiently solve RUTLCSCP-LA-RC, with the computation time less then 10 seconds.

In summary, a set of 10125 instances are generated and solve with good quality and acceptable time. Up to 74.10\% (7502 instances) are solved to optimality, 3.29\% (333 instances) are under-approximation, and 22.62\% (2290 instances) are with no solution.

\vspace*{-1mm}
\section{Conclusion}\label{sec5}
\vspace*{-1mm}

In this paper, we consider an extension of the set covering problem (SCP) called the robust uncertain two-level cooperative set covering problem (RUTLCSCP) by the integration of uncertainty in covering demand nodes. The concepts of probabilistic, robust optimization, and cooperative covering are combined and a compact mixed-integer linear programming (MILP) formulation for the RUTLCSCP is proposed. Computational experiments demonstrates that the RUTLCSCP can be efficiently solved with optimal solutions and a few under-approximate solutions.

In the future, over-approximate solutions with more constraints and less feasible region are likely to investigated. Besides, new exact or heuristic algorithms, new reformulation, and multi-level of the model can be considered. Meanwhile, the proposed model can be applied in many other real-world applications, e.g., collaborative task assignment \cite{xu2020bi}, joint allocation of heterogeneous stochastic resources \cite{wang2019heuristic}, etc.





\end{document}